\let\accentvec\vec
\documentclass[smallextended,referee,envcountsect]{svjour3}
\smartqed

\usepackage{graphicx}
\usepackage{amssymb}
\usepackage[numbers,sort&compress]{natbib}

\journalname{JOTA}

\let\vec\accentvec
\usepackage{amsmath}
\usepackage[colorlinks, citecolor=blue]{hyperref}
\newtheorem{assumption}{Assumption}[section]
\newtheorem{alemma}{Lemma}
\usepackage{anyfontsize}
\usepackage{bookmark}
\makeatletter
\newcommand\subparagraph{%
  \@startsection{subparagraph}{5}
  {\parindent}
  {3.25ex \@plus 1ex \@minus .2ex}
  {-1em}
  {\normalfont\normalsize\bfseries}}
\makeatother
\usepackage{titlesec}
\let\subparagraph\relax
\usepackage{titlesec}
\titleformat{\section}{\normalsize\bfseries}{\thesection}{1em}{}
\titleformat{\subsection}{\normalsize\bfseries}{\thesubsection}{1em}{}

\begin{document}

\title{Kurdyka-{\L}ojasiewicz Property of Zero-Norm Composite Functions}


\author{Yuqia Wu   \and  Shaohua Pan \and Shujun Bi}

\author{Yuqia Wu\textsuperscript{1} \and Shaohua Pan\textsuperscript{1} \and Shujun Bi\textsuperscript{1} }

\institute{Yuqia Wu \\
            math$\_$wuyuqia@mail.scut.edu.cn \\
          Shaohua Pan  \\
            shhpan@scut.edu.cn
            \\
           Shujun Bi (Corresponding author)\\
              bishj@scut.edu.cn \\
              $^1$ \ School of Mathematics, South China University of Technology, Guangzhou, China.}

\date{Received: date / Accepted: date}

\maketitle

\begin{abstract}
  This paper focuses on a class of zero-norm composite optimization problems.
 For this class of nonconvex nonsmooth problems, we establish the Kurdyka-{\L}ojasiewicz
 property of exponent being a half for its objective function under a suitable assumption,
 and provide some examples to illustrate that such an assumption is not very restricted
 which, in particular, involve the zero-norm regularized or constrained piecewise
 linear-quadratic function, the zero-norm regularized or constrained logistic regression function,
 the zero-norm regularized or constrained quadratic function over a sphere.

\end{abstract}
\keywords{KL property of exponent $1/2$ \and zero-norm \and composite optimization}
\subclass{90C26 \and  90C56 \and 49J50}


\section{Introduction}\label{sec1}

 The KL property is an important tool in analysis of optimization,
 dynamic system, partial differential equations, and other directions of applied
 mathematics (see the review paper \cite{Bolte09} and the references therein).
 From \cite[Section 4]{Attouch10}, any subanalytic functions, even more
 general functions with O-minimal structure, automatically satisfy the KL property.
 For the past several years, it has witnessed that the successful application of the KL property
 in analyzing the global convergence of first-order methods for nonconvex
 nonsmooth optimization problems (see, e.g., \cite{Attouch10,Attouch13,Bolte14}).
 In particular, the KL property of exponent 1/2 plays a crucial role in achieving
 the linear convergence rate. As discussed recently in \cite{PanLiu19}
 (see also \cite{Bolte17,WangY18}), for primal lower nice functions,
 the KL property of exponent $1/2$ is usually weaker than the metric subregularity
 of their subdifferential operators \cite{Artacho08} or the Luo-Tseng error bound
 \cite{Tseng09}, which are the common regularity to achieve the linear convergence
 of first-order methods (see, e.g., \cite{Luo92,WenChen17,Zhou17,Cui18}).
 Thus, a valuable research direction is to identify which class of functions
 precisely possesses the KL property of exponent $1/2$.

 Although many classes of functions indeed satisfy the KL property,
 it is not an easy task to estimate the exponent of KL property,
 especially to verify whether they have the KL property of exponent $1/2$.
 Recently, some positive progress have been made in this direction;
 for example, there are some prominent results on how to estimate
 the exponents of KL property in polynomial setting or more general
 semi-algebraic setting (see \cite{Acunto05,Li15,Li18}), and
 some important calculation rules have been developed in \cite{LiPong18,Yu19}
 to identify the exponent of KL property for composite functions
 in terms of the exponent of KL property for its components.
 Although it has explained in \cite{Acunto05} that almost all
 twice continuously differentiable functions have the KL property of
 exponent $1/2$, the deterministic conclusion still requires analysis
 of case by case, especially for those difficult nonconvex nonsmooth functions.
 We notice that Liu et al. \cite{Liu18} established a restricted-type
 KL property of exponent $1/2$ for the quadratic function over orthogonal constraints,
 and Zhang et al. \cite{Zhang18} verified the KL property of exponent 1/2 for several
 classes of regularized matrix factorization functions in the set of their global optima.
 In this work, we focus on the KL property of exponent of $1/2$
 for a class of zero-norm composite functions.

\section{Model and Main Contribution}\label{sec2}
 Let $f:\mathbb{R}^{p}\to\mathbb{R}$ be a smooth function,
 and let $\theta:\mathbb{R}^p\to]-\infty,\infty]$ be a closed proper function.
 We are interested in the following composite problem
 \begin{equation}\label{Theta}
  \min_{x\in\mathbb{R}^p}\Big\{\Theta(x):=f(x)+\theta(x)+h(x)\Big\}
 \end{equation}
 with $h(x):=\nu\|x\|_0$ or $h(x):=\delta_{\Omega}(x)$ for $x\in\mathbb{R}^p$,
 where $\nu>0$ is a regularization parameter, $\|x\|_0$ is the zero-norm
 (cardinality) of $x$, and $\delta_{\Omega}(\cdot)$ denotes the indicator function
 of $\Omega:=\{x\in\mathbb{R}^p:\ \|x\|_0\le\kappa\}$ for a positive integer $\kappa$.

  Since the minimization of $\theta+h$ can be used to capture the structured sparsity,
 the problem \eqref{Theta} has some important applications in a host of fields
 such as statistics, machine learning, signal and image processing, biology,
 and so on. A typical application is the sparse eigenvalue problem such as
 the sparse PCA (see, e.g., \cite{Zou06,Journee10,Yuan13,Asteris14}), for which $\theta$
 is taken as the indicator function of the (nonnegative) unit sphere.
 Another one is the sparse portfolio problem (see, e.g., \cite{Brodie09}),
 for which $\theta$ corresponds to the indicator function of a simplex set.
 In addition, this model also appears as a module in some matrix factorization
 algorithms for nonnegative low rank optimization problems \cite{Zhang17}.

 The main contribution of this work is to establish the KL property of
 exponent $1/2$ for the zero-norm composite function $\Theta$.
 In Section \ref{sec4}, by exploiting the structure of the zero-norm
 function $h$, we show that $\Theta$ is the KL function of exponent $1/2$
 whenever the associated proper lsc function $\theta$ satisfies Assumption \ref{assump1},
 and illustrate that this assumption can be satisfied by several classes of
 functions such as the zero-norm regularized or constrained piecewise linear-quadratic
 functions, the zero-norm regularized or constrained logistic regression function,
 the zero-norm regularized or constrained quadratic optimization problem over a sphere,
 and so on. It is worthwhile to point out that some zero-norm regularized
 and constrained optimization problems were discussed in \cite[Section 5]{Attouch13}
 and \cite[Section 4]{Bolte14}, but the KL property of exponent $1/2$ was not provided there.
 Since the function $h=\delta_{\Omega}$ can be represented as the minimum of finitely
 many proper closed polyhedral functions, when $f$ is convex quadratic function
 and $\theta$ is polyhedral, the KL property of exponent $1/2$
 of $\Theta$ is immediate by \cite[Corollary 5.1]{LiPong18}; when $f$ is a general
 quadratic function and $\theta$ is polyhedral, though the function $\Theta$
 associated to $h=\nu\|\cdot\|_0$ can be represented as the form of \cite[Equation(35)]{LiPong18}
 and \cite[Corollary 5.2]{LiPong18} can be used to identify its KL property
 of exponent $1/2$, the reformulated function is different from $\Theta$
 since they may have different critical point sets.

 \section{Notation and Preliminaries}\label{sec3}

  Throughout this paper, $\mathbb{R}^p$ denotes the $p$-dimensional Euclidean space.
  For a given $\overline{x}\in\mathbb{R}^p$ and $\delta>0$, $\mathbb{B}(\overline{x},\delta)$
  denotes the closed ball centered at $\overline{x}$ of radius $\delta$;
  and for a set $C\subseteq\mathbb{R}^p$, $\delta_C(\cdot)$ means
  the indicator function of $C$, and when $C$ is convex,
  $C^{\infty}$ denotes the recession cone of $C$.
  For an extended real-valued $f:\mathbb{X}\to]-\infty,\infty]$,
  write $[\alpha\le f\le\beta]:=\{x\in\mathbb{X}:\ \alpha\le f(x)\le \beta\}$
  for $\alpha,\beta\in\mathbb{R}$, and say that $f$ is proper if
  ${\rm dom}f$ is nonempty.
  The notation $x'\xrightarrow[f]{}x$ to signify $x'\to x$ and $f(x')\to f(x)$.
  For a vector $x$, $[\![x]\!]$ denotes the subspace generated by $x$,
  and $[\![x]\!]^{\perp}$ means its orthogonal complement.
  For an $m\times n$ matrix $H$ and index sets $I\subseteq\{1,\ldots,m\}$
  and $J\subseteq\{1,\ldots,n\}$, $H_{\!J}$ means the matrix
  consisting of those columns $H_j$ for $j\in J$,
  and $H_{I\!J}$ means the matrix consisting of those entries
  $H_{ij}$ with $(i,j)\in I\times J$. The notation $\mathcal{S}$
  and $E$ denote a unit sphere and an identity matrix whose dimensions
  are known from the context.

 \subsection{Generalized Subdifferentials}\label{sec3.1}

  We first recall several subdifferential notions needed in the subsequent sections.
  The reader can find more information and references in the books \cite{RW98,Mordu06}.
  \begin{definition}\label{Gsubdiff-def}
  Consider a function $f\!:\mathbb{R}^p\to]-\!\infty,\infty]$ and a point
  $x\in{\rm dom}f$, the regular subdifferential of $f$ at $x$ is defined as
  \[
    \widehat{\partial}\!f(x):=\bigg\{v\in\mathbb{R}^p: \
    \liminf_{x'\to x,x'\ne x}
    \frac{f(x')-f(x)-\langle v,x'-x\rangle}{\|x'-x\|}\ge 0\bigg\};
  \]
  the (basic) subdifferential (also known as the limiting
  or Mordukhovich subdifferential) of the function $f$ at $x$ is defined as
  \[
    \partial\!f(x):=\Big\{v\in\mathbb{R}^p:\  \exists\,x^k\xrightarrow[f]{}x\ {\rm and}\
    v^k\in\widehat{\partial}\!f(x^k)\ {\rm with}\ v^k\to v\ {\rm as}\ k\to\infty\Big\};
  \]
  and the horizon subdifferential (also known as the singular subdifferential)
  of the function $f$ at $x$ is defined as
 \[
    \partial^{\infty}f(x)\!:=\!\Big\{v\in\mathbb{R}^p:\  \exists\,x^k\xrightarrow[f]{}x,
    \lambda^k\downarrow 0\ {\rm and}\ v^k\!\in\widehat{\partial}\!f(x^k)\ {\rm s.t.}\ \lambda^kv^k\to v\ {\rm as}\ k\to\infty\Big\}.
  \]
 \end{definition}

  \begin{remark}\label{remark-Fsubdiff}
   Let $\{(x^k,v^k)\}_{k\in\mathbb{N}}$ be a sequence in graph
  ${\rm gph}\partial\!f$ that converges to $(x,v)$ as $k\to\infty$.
  By Definition \ref{Gsubdiff-def}, if $f(x^k)\to f(x)$ as $k\to\infty$,
  then $(x,v)\in{\rm gph}\partial\!f$. The point $\overline{x}$ at which
  $0\in\partial\!f(\overline{x})$ is called a (basic) critical point of $f$.
  In the sequel, we denote by ${\rm crit}f$ the set of critical points of $f$.
 \end{remark}

  Let $C\subseteq\mathbb{R}^p$ be a closed set. By \cite[Exercise 8.14]{RW98},
  the (regular) subdifferential of the indicator function $\delta_C$ at a point
  $\overline{x}\in C$ is precisely the (regular) normal cone to
  $C$ at $\overline{x}$. For the definitions of the regular normal cone
  $\widehat{\mathcal{N}}_C(\overline{x})$, the normal cone
  $\mathcal{N}_C(\overline{x})$ and the proximal normal cone
  $\widetilde{\mathcal{N}}_C(\overline{x})$ of $C$ at $\overline{x}$,
  please refer to \cite[Chapter 6]{RW98}. The following lemma provides
  the generalized subdifferential characterizations of $\delta_{\mathcal{S}}$.
  Since its proof is easily obtained by using
  \cite[Exercise 6.7\&Exercise 8.14]{RW98}, we here omit it.
 \begin{lemma}\label{Gsubdiff-sphere}
  For any $\overline{x}\in\mathcal{S}$, the unit sphere,
  the following equalities hold:
  \begin{equation*}
   \widehat{\partial}\delta_{\mathcal{S}}(\overline{x})
   =\partial\delta_{\mathcal{S}}(\overline{x})
   =\{\omega\overline{x}:~\omega \in \mathbb{R}\}
   =\partial^{\infty}\delta_{\mathcal{S}}(\overline{x})
   =\big[\widehat{\partial}\delta_{\mathcal{S}}(\overline{x})\big]^{\infty}.
  \end{equation*}
 \end{lemma}

 \subsection{Generalized Subdifferentials of \texorpdfstring{$h$}{h}}\label{sec3.2}

  First we provide the generalized subdifferentials of the zero-norm.
  Its (regular) subdifferentials are seen in \cite{Le13}.
  Here we supplement its horizon subdifferential, which along with the result of
  \cite{Le13} implies its regularity.
 \begin{lemma}\label{Gsubdiff-znorm}
  Let $h(x)=\nu\|x\|_0$ for $x\in\mathbb{R}^p$. Fix an arbitrary
  $\overline{x}\in\mathbb{R}^p$. Then,
  \[
   \widehat{\partial}h(\overline{x}) = \partial h(\overline{x})
   =\big\{\xi\in\mathbb{R}^p\!:~\xi_i=0 ~{\rm for}~ i\in{\rm supp}(\overline{x})\big\}
   =\partial^{\infty}h(\overline{x})=[\widehat{\partial} h(\overline{x})]^{\infty}.
  \]
 \end{lemma}
 \begin{proof}
  The first two equalities hold by \cite[Theorem 1]{Le13}.
  Write $J={\rm supp}(\overline{x})$ and
  $\Xi=\big\{\xi\in\mathbb{R}^p:\ \xi_J = 0\big\}$.
  We next prove that $\Xi=\partial^{\infty} h(\overline{x})$,
  i.e., the third equality holds.
  Let $\overline{v}\in\partial^{\infty} h(\overline{x})$.
  By Definition \ref{Gsubdiff-def}, there exist $x^k\xrightarrow[h]{}\overline{x}$,
  $\lambda^k \downarrow 0$ and $v^k \in \widehat{\partial}h(x^k)$ such that
  $\lambda^kv^k \rightarrow \overline{v}$ as $k\to\infty$. From $x^k\xrightarrow[h]{}\overline{x}$,
  it follows that ${\rm supp}(x^k)=J$ for all sufficiently large $k$.
  Indeed, from $x^k \rightarrow \overline{x}$ we have
  ${\rm supp}(x^k) \supseteq {\rm supp}(\overline{x})$, which along with
  $h(x^k) \rightarrow h(\overline{x})$ implies that ${\rm supp}(x^k)={\rm supp}(\overline{x})=J$.
  Then, $v_J^k=0$ for all large enough $k$.
  and $v_J^k=0$ for all large enough $k$.
  Along with $\lambda^kv_J^k \rightarrow \overline{v}_J$, we have $\overline{v}_J=0$.
  Then, $\partial^{\infty} h(\overline{x})\subseteq\Xi$.
  Conversely, take an arbitrary $\overline{v}\in\Xi$.
  Let $x^k=\overline{x}$, $\lambda^k=\frac{1}{k}$ and $v^k = k\overline{v}$ for each $k$.
  Clearly, $x^k\xrightarrow[h]{}\overline{x}$ and $v^k\in\widehat{\partial}h(x^k)$
  with $\lambda^kv^k \rightarrow \overline{v}$. So, $\overline{v}\in\partial^{\infty}h(\overline{x})$
  and $\Xi\subseteq\partial^{\infty} h(\overline{x})$ follows. Thus, $\Xi=\partial^{\infty}h(\overline{x})$.
 Recall that $\widehat{\partial} h(\overline{x})$ is closed and convex.
 Since $0 \in \widehat{\partial} h(\overline{x})$ and
 $tv \in \widehat{\partial} h(\overline{x})$ for any $v\in\widehat{\partial} h(\overline{x})$
 and $t \geq 0$, by \cite[Theorem 8.3]{Roc70} we have
 $\widehat{\partial} h(\overline{x})=[\widehat{\partial} h(\overline{x})]^{\infty}$.
 The last equality holds. \qed
 \end{proof}

  The following lemma provides the generalized subdifferentials of $h=\delta_{\Omega}$
 at $\overline{x}\in\Omega$. Since its proof can be found in \cite{Bauschke14},
 here we omit it.
 \begin{lemma}\label{Ncone-L0ball}
  Fix an arbitrary $\overline{x}\in\Omega$. Let $J={\rm supp}(\overline{x})$
  and $\overline{J}=\{1,\ldots,p\}\backslash J$.
  \begin{itemize}
   \item[(i)] If $\|\overline{x}\|_0=\kappa$, then
               \(
                \widetilde{\mathcal{N}}_{\Omega}(\overline{x})
                =\widehat{\mathcal{N}}_{\Omega}(\overline{x})
                =\big\{v\in\mathbb{R}^p:\ v_J=0\big\}
                =\mathcal{N}_{\Omega}(\overline{x}).
              \)

   \item [(ii)] If $\|\overline{x}\|_0<\kappa$, then
                \(
                 \widetilde{\mathcal{N}}_{\Omega}(\overline{x})
                 =\{0\}=\!\widehat{\mathcal{N}}_{\Omega}(\overline{x})\subseteq\!
                  \mathcal{N}_{\Omega}(\overline{x})=\Gamma
                 \)
                with $\Gamma$ defined by
                \begin{equation}\label{Gamma}
                  \Gamma:=\big\{v\in\mathbb{R}^p:\ \exists \widehat{J}\subseteq \overline{J}\
                 \ {\rm with}\ |\widehat{J}|=\kappa-\!|J|\ {\rm such\ that}\ v_{J\cup\widehat{J}}=0\big\}.
                \end{equation}
  \end{itemize}
  \end{lemma}

 \subsection{Regular Zero-Norm Composite Functions}\label{sec3.3}

 We first argue that $h+\delta_{\mathcal{S}}$ is regular,
 which requires the following lemma.
 \begin{lemma}\label{Gsubdiff-sphere1}
  Let $\psi\!:\mathbb{R}^p\to]-\infty,+\infty]$ be a proper lsc function.
  Consider an arbitrary point $\overline{x}\in{\rm dom}\psi\cap\mathcal{S}$.
  If $\psi$ is regular at $\overline{x}$ and
  $\partial^{\infty}\psi(\overline{x})\subseteq[\![\overline{x}]\!]^{\perp}$, then
  \[
   \widehat{\partial}(\psi+\delta_{\mathcal{S}})(\overline{x})
   \!=\partial(\psi+\delta_{\mathcal{S}})(\overline{x})
   =\partial\psi(\overline{x})+\partial\delta_{\mathcal{S}}(\overline{x})
   \!=\partial^{\infty}(\psi+\delta_{\mathcal{S}})(\overline{x})
   \!=[\widehat{\partial}(\psi+\delta_{\mathcal{S}})(\overline{x})]^{\infty}.
  \]
 \end{lemma}
 \begin{proof}
  Let $u\in\partial^{\infty}\psi(\overline{x})$ and
  $v\in\partial^{\infty}\delta_{\mathcal{S}}(\overline{x})$
  be such that $u+v=0$. By Lemma \ref{Gsubdiff-sphere}, there exists $\omega\in\mathbb{R}$
  such that $v=\omega\overline{x}$, and hence
  \(
    u+\omega \overline{x}=0.
  \)
  Since $\partial^{\infty}\psi(\overline{x})\subseteq[\![\overline{x}]\!]^{\perp}$,
  we have $\langle u,\overline{x}\rangle=0$. Together with $u+\omega \overline{x}=0$
  and $\overline{x}\in\mathcal{S}$, we get $\omega=0$,
  and then $u=v=0$. The result follows by \cite[Corollary 10.9]{RW98}. \qed
  \end{proof}

  By Lemma \ref{Gsubdiff-znorm} and \ref{Ncone-L0ball}, the assumption of
 Lemma \ref{Gsubdiff-sphere1} is satisfied by the functions
 $\psi(\cdot)=\nu\|\cdot\|_0$ at any $x\in\mathbb{R}^p$
 and $\psi=\delta_{\Omega}$ at those $x\in\Omega$ with $\|x\|_0=\kappa$.
 Then, from Lemma \ref{Gsubdiff-sphere1},
 we immediately get the following result.
 \begin{proposition}\label{prop31}
  If $h(x)=\nu\|x\|_0$ for $x\in\mathbb{R}^p$,
  then for any $\overline{x}\in\mathcal{S}$ we have
  \[
    \widehat{\partial}(\delta_{\mathcal{S}}\!+h)(\overline{x})
    =\partial(\delta_{\mathcal{S}}\!+h)(\overline{x})
    =\delta_{\mathcal{S}}(\overline{x})\!+\partial h(\overline{x})
    =\partial^{\infty}(\delta_{\mathcal{S}}\!+h)(\overline{x})
    =[\widehat{\partial}(\delta_{\mathcal{S}}\!+h)(\overline{x})]^{\infty}.
 \]
 If $h=\delta_{\Omega}$, for any $\overline{x}\in\mathcal{S}$ with
 $\|\overline{x}\|_0=\kappa$, the last equalities hold;
 and for any $\overline{x}\in\mathcal{S}$
 with $\|\overline{x}\|_0<\kappa$, it holds that
 $\partial(\delta_{\mathcal{S}}+h)(\overline{x})
   \subseteq \partial\delta_{\mathcal{S}}(\overline{x})+\partial h(\overline{x})$.
 \end{proposition}

  The following proposition states which class of proper closed convex
  functions $\psi$ is such that $\psi+h$ is regular, whose proof is found
  in Appendix C. When $h(\cdot)=\nu\|\cdot\|_0$, this proposition extends
  the result of \cite[Lemma 3.3]{Feng19}.
   \begin{proposition}\label{prop32}
  {\bf(i)} When $h(\cdot)=\nu\|\cdot\|_0$, if $\psi\!:\mathbb{R}^p\to]-\infty,+\infty]$
  is a proper closed piecewise linear function, then
  for any $\overline{x}\in\mathbb{R}^p$ with $\partial\psi(\overline{x})\ne\emptyset$,
  \begin{align}\label{first-group}
    \widehat{\partial}(\psi\!+\!h)(\overline{x})
     &=\partial\psi(\overline{x})\!+\!\partial h(\overline{x})
     =\partial(\psi\!+\!h)(\overline{x}),\\
     \partial^{\infty}(\psi\!+\!h)(\overline{x})
    &=\partial^{\infty}\psi(\overline{x})+\partial^{\infty} h(\overline{x})
    =[\partial\psi(\overline{x})\!+\!\partial h(\overline{x})]^{\infty};
  \end{align}
  if $\psi$ is an indicator of some closed convex set $C\subseteq\mathbb{R}^p$,
  then for any $\overline{x}\in C$ with ${\rm ri}(C)\cap\{x\in\mathbb{R}^p\,|\,
  x_i=0\ {\rm for}\ i\notin{\rm supp}(\overline{x})\}\ne\emptyset$,
  the last equalities also hold.
  \[
    \widehat{\partial}(\psi\!+\!h)(\overline{x})
    =\partial\psi(\overline{x})\!+\!\partial h(\overline{x})
    =\!\partial(\psi\!+\!h)(\overline{x})
    =\!\partial^{\infty}(\psi\!+\!h)(\overline{x})
    =[\widehat{\partial}(\psi\!+\!h)(\overline{x})]^{\infty};
  \]

  \noindent
  {\bf(ii)} When $h=\delta_{\Omega}$, these equalities hold at any $\overline{x}\in{\rm dom}\psi$
  with $\|\overline{x}\|_0=\kappa$; and at any $\overline{x}\in{\rm dom}\psi$
  with $\|\overline{x}\|_0<\kappa$ it holds that
  $\partial(\psi+h)(\overline{x})
  \subseteq\partial\psi(x)+\partial h(\overline{x})$.
  \end{proposition}
  \begin{remark}
   When $\psi$ is a locally Lipschitz regular function, the first part of
   Proposition \ref{prop32} still holds by invoking \cite[Theorem 9.13(b) \& Corollary 10.9]{RW98}.
  \end{remark}

 \subsection{Kurdyka-{\L}ojasiewicz Property}\label{sec3.4}

 \begin{definition}\label{KL-Def1}
  Let $f\!:\mathbb{R}^p\!\to]-\!\infty,\infty]$ be a proper function.
  The function $f$ is said to have the Kurdyka-{\L}ojasiewicz (KL) property
  at $\overline{x}\in{\rm dom}\,\partial\!f$ if there exist $\eta\in]0,\infty]$,
  a continuous concave function $\varphi\!:[0,\eta[\to\mathbb{R}_{+}$ satisfying
  \begin{itemize}
    \item [(i)] $\varphi(0)=0$ and $\varphi$ is continuously differentiable on $]0,\eta[$;

    \item[(ii)] for all $s\in]0,\eta[$, $\varphi'(s)>0$,
  \end{itemize}
  and a neighborhood $\mathcal{U}$ of $\overline{x}$ such that for all
  \(
    x\in\mathcal{U}\cap\big[f(\overline{x})<f<f(\overline{x})+\eta\big],
  \)
  \[
    \varphi'(f(x)-f(\overline{x})){\rm dist}(0,\partial\!f(x))\ge 1.
  \]
  If $\varphi$ can be chosen as $\varphi(s)=c\sqrt{s}$
  for some $c>0$, then $f$ is said to have the KL property at $\overline{x}$
  with an exponent of $1/2$. If $f$ has the KL property of exponent $1/2$
  at each point of ${\rm dom}\,\partial\!f$,
  then $f$ is called a KL function of exponent $1/2$.
 \end{definition}
 \begin{remark}\label{KL-remark}
  To show that a proper function is a KL function of exponent $1/2$,
  it suffices to verify if it has the KL property of exponent $1/2$
  at all critical points since, by \cite[Lemma 2.1]{Attouch10},
  it has this property at all noncritical points.
 \end{remark}

 \section{Kurdyka-{\L}ojasiewicz Property of Exponent 1/2 of \texorpdfstring{$\Theta$}{Theta}}\label{sec4}
  In this section, we shall establish the KL property of exponent $1/2$
  for the function $\Theta$ on its critical point set under the following
  assumption on $\theta$:
 \begin{assumption}\label{assump1}
  The proper lsc function $\theta$ satisfies the following conditions:
 \begin{itemize}
  \item[(i)] $\theta$ is continuous relative to the set ${\rm dom}\,\partial\theta$;

  \item[(ii)] $\theta$ is regular at every point of ${\rm dom}\,\partial\theta$;

  \item[(iii)] for every $x\in{\rm dom}\partial\theta$,
              $\partial(\theta+h)(x)\subseteq\partial\theta(x)+\partial h(x)$;

  \item[(iv)] for every $I\subset\{1,\ldots,p\}$, $g_{I}:=f_{I}+\theta_{I}$
               is a KL function of exponent $1/2$, where
               $f_{\!I}(z):=f(E_{I}z)$ and $\theta_{I}(z)\!:=\theta(E_{I}z)$
               for $z\in\mathbb{R}^{|I|}$.
  \end{itemize}
 \end{assumption}

 Assumption \ref{assump1}(i)-(iii) are the common requirement in dealing with
 nonsmooth functions, and Assumption \ref{assump1}(iv) seems to be a little
 more restricted. In the sequel, we provide several classes of
 examples to satisfy this condition.

 We first achieve the KL property of exponent $1/2$ of $\Theta$ with $h\equiv\nu\|\cdot\|_0$.
 \begin{theorem}\label{KLznorm-Theta}
  Suppose that $h\equiv\nu\|\cdot\|_0$ and Assumption \ref{assump1} holds.
  Then the function $\Theta$ has the KL property of exponent $1/2$
  at all critical points.
 \end{theorem}
 \begin{proof}
  Fix an arbitrary $\overline{x}\in{\rm crit}\Theta$. Let $J:={\rm supp}(\overline{x})$,
  and $g_{\!J}$ be defined as in Assumption \ref{assump1}(iv).
  Let $g\equiv f+\theta$. Obviously, $g_{\!J}(z)=g(E_{\!J}z)$ for $z\in\mathbb{R}^{|J|}$.
  Since $g_{\!J}$ is a KL function of exponent $1/2$, there exist
  $\delta_1>0,\eta_1>0$ and $c_{1}>0$ such that
  for all $z\in\mathbb{B}(\overline{x}_{\!J},\delta_1)
  \cap[g_{\!J}(\overline{x}_{\!J})<g_{\!J}<g_{\!J}(\overline{x}_{\!J})+\eta_1]$,
  \begin{equation}\label{KL-gfun}
   \mbox{dist}(0,\partial g_{\!J}(z))
   \ge c_1\sqrt{g_{\!J}(z)-g_{\!J}(\overline{x}_{\!J})}\,.
  \end{equation}
  Take $\eta_2\in]0,{\nu}/{3}[$. By Assumption \ref{assump1}(i),
  there exists $\delta_2>0$ such that
  \begin{equation}\label{continuity1}
    |g(x)-g(\overline{x})|<\eta_2
    \quad\ \forall x\in\mathbb{B}(\overline{x},\delta_2)\cap{\rm dom}\partial\theta.
  \end{equation}
  Take $\delta=\min(\delta_1,\delta_2)$ and $\eta=\min(\eta_1,\eta_2)$.
  Pick an arbitrary $x$ from the set
  $\mathbb{B}(\overline{x},\delta)\cap[\Theta(\overline{x})<\Theta<\Theta(\overline{x})+\eta]$.
  We proceed the arguments by two cases.

  \noindent
  {\bf Case 1:} $x\in{\rm dom}\partial\Theta$. By the expression of $\Theta$
  and \cite[Exercise 8.8]{RW98}, we have
  \begin{equation}\label{subdiff1-Theta}
    \partial\Theta(x)
    \subseteq \nabla\!f(x)+\partial\theta(x)+\nu\partial\|\cdot\|_0(x)
    =\partial g(x)+\nu\partial\|\cdot\|_0(x)
  \end{equation}
  where the inclusion is also by Assumption \ref{assump1}(iii).
  This means that $x\in{\rm dom}\partial\theta$.
  Also, we have $g(x)>g(\overline{x})$ (if not, by combining
  $g(x)\le g(\overline{x})$ and \eqref{continuity1} with
  $\Theta(\overline{x})<\Theta(x)<\Theta(\overline{x})+\eta$,
  one may obtain a contradiction
  $\|\overline{x}\|_0\leq\|x\|_0-1<\|\overline{x}\|_0+\frac{1}{\nu}(\eta +\eta_2)-1<\|\overline{x}\|_0$).
  Together with $\Theta(x)<\Theta(\overline{x})+\eta$,
  we deduce that $\|x\|_0\le\|\overline{x}\|_0$.
  In addition, by reducing $\delta$ if necessary, we also have $\|x\|_0\ge\|\overline{x}\|_0$.
  Thus, $\|x\|_0 = \|\overline{x}\|_0$. Notice that ${\rm supp}(x)\supseteq {\rm supp}(\overline{x})$
  (if necessary by shrinking the value of $\delta$). Hence,
  the following relation holds:
  \begin{equation}\label{indx-equa}
    {\rm supp}(x)={\rm supp}(\overline{x})=J.
  \end{equation}
  Now by invoking \eqref{subdiff1-Theta} and Lemma \ref{Gsubdiff-znorm},
  there exists $\zeta^*\in\partial g(x)$ such that
  \begin{align}\label{subdiff-fg}
   \mbox{dist}(0, \partial \Theta(x))
   &\ge \mbox{dist}\big(0,\partial g(x)\!+\nu\partial\|\cdot\|_0(x)\big)\nonumber\\
   &=\min_{\zeta\in\partial g(x),\xi\in\nu\partial\|\cdot\|_0(x)}\!\|\zeta+\xi\|
   =\|\zeta_J^*\|.
  \end{align}
  Recall that $\theta_{\!J}(z)\equiv\theta(E_{\!J}z)$ for $z\in\mathbb{R}^{|J|}$.
  By \cite[Theorem 10.6]{RW98} and Assumption \ref{assump1}(ii),
  $\widehat{\partial}\theta_{\!J}(x_{\!J})\supseteq
  E_{\!J}^{\mathbb{T}}\widehat{\partial}\theta(E_{\!J}x_{\!J})
  =E_{\!J}^{\mathbb{T}}\partial\theta(E_{\!J}x_{\!J})$.
  From $g_{\!J}(z)\!=f(E_{\!J}z)+\theta(E_{\!J}z)$,
  \begin{align*}
   \partial g_{\!J}(x_{\!J})&\supseteq E_{\!J}^{\mathbb{T}}\nabla\!f(E_{\!J}x_{\!J})
    +\widehat{\partial}\theta_{\!J}(x_{\!J})
    \supseteq E_{\!J}^{\mathbb{T}}\nabla\!f(E_{\!J}x_{\!J})
    +E_{\!J}^{\mathbb{T}}\partial\theta(E_{\!J}x_{\!J})\\
    &=E_{\!J}^{\mathbb{T}}[\nabla\!f(E_{\!J}x_{\!J})+\partial\theta(E_{\!J}x_{\!J})]
    =E_{\!J}^{\mathbb{T}}\partial g(E_{\!J}x_{\!J})
  \end{align*}
  where the first inclusion and the last equality is due to \cite[Exercise 8.8]{RW98}.
  In addition, from $\zeta^*\in\partial g(x)$, we have
  $\zeta_{J}^*\in E_{\!J}^{\mathbb{T}}\partial g(x)
    =E_{\!J}^{\mathbb{T}}\partial g(E_{\!J}x_{\!J})$.
 From the last equation, it follows that $\zeta_{\!J}^*\in\partial g_{\!J}(x_{\!J})$.
  Thus, along with \eqref{subdiff-fg}, we obtain
  \begin{equation}\label{subdiff-gJ1}
    \mbox{dist}(0, \partial \Theta(x))
    \ge\|\zeta_J^*\|\ge{\rm dist}(0,\partial g_{\!J}(x_{\!J})).
  \end{equation}
  Recall that $x\in[\Theta(\overline{x})<\Theta<\Theta(\overline{x})+ \eta]$.
  By invoking \eqref{indx-equa}, it follows that
  \[
    \Theta(x)=g(E_{\!J}x_{\!J})+|J|=g_{\!J}(x_{\!J})+|J|
    \ \ {\rm and}\ \
    \Theta(\overline{x})=g_{\!J}(\overline{x}_{\!J})\!+\!|J|.
  \]
  Thus, $x_{\!J}\in[g_{\!J}(\overline{x}_{\!J})<g_{\!J}<g_{\!J}(\overline{x}_{\!J})+\eta_1]$.
  Since $x_{\!J}\in\mathbb{B}(\overline{x}_{\!J},\delta_1)$,
  by \eqref{subdiff-gJ1} and \eqref{KL-gfun},
  \[
    \mbox{dist}(0,\partial\Theta(x))
    \ge{\rm dist}(0,\partial g_{\!J}(x_{\!J}))
    \geq c_1\sqrt{g_{\!J}(x_{\!J})\!-\!g_{\!J}(\overline{x}_{\!J})}
    =c_1\sqrt{\Theta(x)\!-\!\Theta(\overline{x})}.
  \]

  \noindent
  {\bf Case 2:} $x\notin{\rm dom}\partial\Theta$. In this case, $\partial\Theta(x)=\emptyset$,
  and $\mbox{dist}(0,\partial\Theta(x))=\infty$. This means that
  the last inequality automatically holds.

  Now by the arbitrariness of $x$ in
  $\mathbb{B}(\overline{x},\delta)\cap [\Theta(\overline{x})<\Theta<\Theta(\overline{x})+ \eta]$,
  the last inequality shows that the function $\Theta$ has the KL property of exponent $1/2$
  at $\overline{x}$. By the arbitrariness of $\overline{x}$ in ${\rm crit}\Theta$,
  the desired result follows.  \qed
 \end{proof}

  By Remark \ref{KL-remark}, Theorem \ref{KLznorm-Theta} shows that
  $\Theta$ with $h\equiv\nu\|\cdot\|_0$ is a KL function of exponent $1/2$
  if the associated $\theta$ satisfies Assumption \ref{assump1}.
  We next illustrate that it can be satisfied
  by several classes of proper lsc functions.
 \begin{example}\label{example41}
  Let $f(x)\!:=x^{\mathbb{T}}Ax$ and $\theta(x):=\delta_{\mathcal{S}}(x)$
  for $x\in\mathbb{R}^p$, where $A$ is a $p\times p$ symmetric matrix.
  Assumption \ref{assump1} (ii)-(iii) holds by Lemma \ref{Gsubdiff-sphere}
  and Proposition \ref{prop31}, respectively.
  For any $I\subseteq\{1,\ldots,p\}$, it is easy to check that
  $g_{I}(z)=2z^{\mathbb{T}}A_{I\!I}z+\delta_{\mathcal{S}}(z)$
  for $z\in\mathbb{R}^{|I|}$. By Lemma \ref{manifold-KL} in Appendix A and
  \cite[Theorem 1]{Liu18}, $g_{I}$ is a KL function of exponent $1/2$. Thus,
  $\Theta$ associated to such $f$ and $\theta$ is the KL function of exponent $1/2$.
  Though the result of \cite[Theorem 1]{Liu18} implies that $g_{\!I}$ is
  a KL function of exponent $1/2$, its proof is not easy to follow for the reader.
  We provide a concise proof in Appendix B.
 \end{example}

 \begin{example}\label{example42}
  Consider $f(x):=\frac{1}{2}x^{\mathbb{T}}Mx+b^{\mathbb{T}}x$ and
  $\theta(x):=\delta_{\mathcal{P}}(x)$ for $x\in\mathbb{R}^p$,
  where $M$ is a $p\times p$ symmetric matrix and $b\in\mathbb{R}^{p}$
  is a vector, and $\mathcal{P}\subseteq\mathbb{R}^p$ is a polyhedral set.
  Assumption \ref{assump1} (iii) holds by Proposition \ref{prop32}.
  For any $I\subseteq\{1,\ldots,p\}$, since $\partial g_{I}$ is
  a polyhedral multifunction, by \cite[Proposition 1]{Robinson81}
  $\partial g_{I}$ is metrically subregular at every point of its graph.
  From \cite[Lemma 3.1]{Luo92} or the proof of \cite[Corollary 5.2]{LiPong18},
  we know that \cite[Assumption 3.1]{PanLiu19} holds, and then $g_{I}$
  is a KL function of exponent $1/2$ by \cite[Theorem 3.1(ii)]{PanLiu19}.
  Thus, $\Theta$ associated to such $f$ and $\theta$ is the KL function of exponent $1/2$.
  It is worthwhile to point out that the result cannot be got by
  using \cite[Corollary 5.2]{LiPong18} since the zero-norm is
  discontinuous relative to ${\rm dom}\partial\Theta$.
  Now by invoking \cite[Corollary 3.1]{LiPong18}, we conclude that
  the following $\Phi_1$ is a KL function of exponent $1/2$:
  \[
    \Phi_1(x):=\min_{1\le i\le m}\Big\{\frac{1}{2}x^{\mathbb{T}}M_ix+b_i^{\mathbb{T}}x
    +\delta_{\mathcal{P}_i}(x)\Big\}+\nu\|x\|_0\quad\forall x\in\mathbb{R}^p,
  \]
  where $M_i,\,i=1,\ldots,m$ are $p\times p$ symmetric matrices,
  and each $\mathcal{P}_i$ is polyhedral.
  \end{example}
 \begin{example}\label{example43}
  Let $f(x)\!:=\!\phi(Ax)$ for $x\in\!\mathbb{R}^p$ where
  $A=\![a_1\ \ldots\ a_n]^{\mathbb{T}}\in\mathbb{R}^{n\times p}$
  and $\phi(z):=\sum_{i=1}^n\log\big[1+\exp(-b_iz_i)\big]$
  with each $b_i\in\{-1,1\}$. Let $\theta(x)\equiv 0$
  for $x\in\mathbb{R}^p$. For each $I\subseteq\{1,\ldots,p\}$,
  by Lemma \ref{composite} in Appendix C, $g_{I}(z)=f(E_{I}z)$
  for $z\in\mathbb{R}^{|I|}$ is a KL function of exponent $1/2$.
  So, the zero-norm regularized logistic
  regression function $\Theta$ has the KL property
  of exponent $1/2$.
 \end{example}
 \begin{example}\label{example44}
  Let $f(x)\equiv 0$ and $\theta(x)\!:=\|Ax\!-b\|_q+\frac{\gamma}{2}\|Ax-\!Au\|^2$
  for $x\in\!\mathbb{R}^p$, where $A\in\mathbb{R}^{n\times p},b\in\mathbb{R}^n$
  and $u\in\mathbb{R}^p$ are the given data, $q\in[1,+\infty]$ is a real number and
  $\gamma>0$ is a parameter. Notice that $\theta(x)\equiv\phi(Ax)$ with
  the strongly convex $\phi(z):=\|z-b\|_q+\frac{\gamma}{2}\|z-\!Au\|^2$
  for $z\in\mathbb{R}^p$. By Lemma \ref{composite},
  for each $I\subseteq\{1,\ldots,p\}$, $g_{I}(z)=\theta(E_{I}z)$ for
  $z\in\mathbb{R}^{|I|}$ is a KL function of exponent $1/2$.
  So, the $\Theta$ associated to such $f$ and $\theta$
  is a KL function of exponent $1/2$.
 \end{example}
 \begin{example}\label{example45}
  Let $f(x)\equiv 0$ and $\theta(x)\!:=\|Ax\!-b\|_q+\frac{\gamma}{2}\|x\|^2+\delta_{\Delta}(x)$
  for $x\in\!\mathbb{R}^p$, where $\Delta\!:=\{x\in\mathbb{R}_{+}^p\!:\,\sum_{i=1}^px_i=1\}$,
  and $A\in\mathbb{R}^{n\times p},b\in\mathbb{R}^n,q,\gamma$ are same as above.
  For each $I\subseteq\{1,\ldots,p\}$, $g_{I}(z)=\theta(E_{I}z)$ for $z\in\mathbb{R}^{|I|}$
  is a closed proper strongly convex function, and is a KL function of
  exponent $1/2$ by Lemma \ref{composite}. The $\Theta$ associated to such $f$
  and $\theta$ is a KL function of exponent $1/2$.
 \end{example}

 Next we focus on the KL property of exponent $1/2$ of $\Theta$ with $h\equiv\delta_{\Omega}$.
 \begin{theorem}\label{KLzconstr-Theta}
  Suppose that $h=\delta_{\Omega}$ and Assumption \ref{assump1} holds.
  Then the function $\Theta$ has the KL property of exponent $1/2$
  at all critical points.
 \end{theorem}
 \begin{proof}
  Fix an arbitrary $\overline{x}\in{\rm crit}\Theta$.
  We proceed the arguments by two cases.

  \medskip
  \noindent
  {\bf Case 1: $\|\overline{x}\|_0=\kappa$}. Let $J:={\rm supp}(\overline{x})$,
  and let $g_{\!J}$ be defined as in Assumption \ref{assump1} (iv).
  Write $g\equiv f+\theta$. Since $g_{\!J}$ is a KL function of exponent $1/2$,
  there exist $\delta>0,\eta>0$ and $c>0$ such that
  for all $z\in\mathbb{B}(\overline{x}_{\!J},\delta)
  \cap[g_{\!J}(\overline{x}_{\!J})<g_{\!J}<g_{\!J}(\overline{x}_{\!J})+\eta]$,
  \begin{equation}\label{KL1-gfun}
  \mbox{dist}(0,\partial g_{\!J}(z)) \geq c\sqrt{g_{\!J}(z)-g_{\!J}(\overline{x}_{\!J})}
  \end{equation}
  Fix an arbitrary $x\in\mathbb{B}(\overline{x},\delta)
  \cap[\Theta(\overline{x})<\Theta<\Theta(\overline{x})+\eta]$.
  Due to Lemma \ref{Ncone-L0ball}(i), by using the same arguments as
  those for Theorem \ref{KLznorm-Theta}, one can get the result.

  \medskip
  \noindent
  {\bf Case 2: $\|\overline{x}\|_0<\kappa$}. Let
  $\mathcal{I}\!:=\!\{I\!:\,\{1,\ldots,p\}\supseteq I\supseteq {\rm supp}(\overline{x})\}$.
  For each $I\in\!\mathcal{I}$, by Assumption \ref{assump1} (iv)
  $g_{I}$ is a KL function of exponent $1/2$. So, there are $\delta_{I}>0$,
  $\eta_{I}>0$ and $c_{I}>0$ such that for all $z\in\mathbb{B}(\overline{x}_{I},\delta_{I})
  \cap[g_{I}(\overline{x}_{I})<g_{I}<g_{I}(\overline{x}_{I})+\eta_{I}]$,
  \begin{equation}\label{KL-gIfun}
  \mbox{dist}(0,\partial g_{I}(z)) \geq c_{I}\sqrt{g_{I}(z)-g_{I}(\overline{x}_I)}.
  \end{equation}
  In addition, by the continuity there exists $\delta_1>0$ such that
  for all $x'\in\mathbb{B}(\overline{x},\delta_1)$,
  ${\rm supp}(x')\supseteq{\rm supp}(\overline{x})$.
  Notice that $\mathcal{I}$ contains a finite number of index sets. Set
  \[
   \delta=\min(\delta_1,\min_{I\in\mathcal{I}}\delta_{I}),\ \eta:=\min_{I\in\mathcal{I}}\eta_{I}
   \ \ {\rm and}\ \ c:=\min_{I\in\mathcal{I}}c_{I}.
  \]
  Pick any $x\in\mathbb{B}(\overline{x},\delta)
  \cap[\Theta(\overline{x})<\Theta<\Theta(\overline{x})+\eta]$.
  Consider the following subcases.

  \noindent
  {\bf Subcase 2.1:} $x\in{\rm dom}\partial\Theta$. Now since
  ${\rm dom}\partial\Theta\subseteq{\rm dom}\Theta$, we have
  $x\in{\rm dom}\theta\cap\Omega$ and
  $J:={\rm supp}(x)\supseteq{\rm supp}(\overline{x})$.
  By \cite[Exercise 8.8]{RW98} and Assumption \ref{assump1}(iii),
  \begin{equation}\label{subdiff-Theta}
    \partial\Theta(x)
    \subseteq \nabla\!f(x)+\partial\theta(x)+\mathcal{N}_{\Omega}(x)
    =\partial g(x)+\mathcal{N}_{\Omega}(x).
  \end{equation}
  By combining \eqref{subdiff-Theta} with Lemma \ref{Ncone-L0ball},
  there exists $\zeta^*\in\partial g(x)$ such that
  \[
   \mbox{dist}(0, \partial \Theta(x))
   \ge \mbox{dist}\big(0,\partial g(x)\!+\mathcal{N}_{\Omega}(x)\big)
   =\min_{\zeta\in\partial g(x),\xi\in\mathcal{N}_{\Omega}(x)}\!\|\zeta+\xi\|
   \ge\|\zeta_J^*\|.
  \]
  In addition, from the proof of Theorem \ref{KLznorm-Theta}, we know that
  $\zeta_{\!J}^*\in\partial g_{\!J}(x_{\!J})$. Thus,
  \begin{equation}\label{subdiff-gJ}
   \mbox{dist}(0, \partial \Theta(x))
   \ge\|\zeta_J^*\|\ge{\rm dist}(0,\partial g_{\!J}(x_{\!J})).
  \end{equation}
  Recall that $x\in[\Theta(\overline{x})<\Theta<\Theta(\overline{x})+ \eta]$.
  From $J={\rm supp}(x)\supseteq{\rm supp}(\overline{x})$, we get
  \begin{equation*}
   \Theta(x)=g(x)=g(E_{\!J}x_{\!J})=g_{\!J}(x_{\!J})\ \ {\rm and}\ \
   \Theta(\overline{x})=g(\overline{x})=g(E_{\!J}\overline{x}_{\!J})
    =g_{\!J}(\overline{x}_{\!J}).
  \end{equation*}
  Thus, $x_{\!J}\in[g_{\!J}(\overline{x}_{\!J})<g_{\!J}<g_{\!J}(\overline{x}_{\!J})+\eta_1]$.
  Since $x_{\!J}\in\mathbb{B}(\overline{x}_{\!J},\delta_1)$,
  by \eqref{subdiff-gJ} and \eqref{KL-gIfun},
  \[
    \mbox{dist}(0,\partial\Theta(x))
    \ge{\rm dist}(0,\partial g_{\!J}(x_{\!J}))
    \ge c_{\!J}\sqrt{g_{\!J}(x_{\!J})\!-\!g_{\!J}(\overline{x}_{\!J})}
    \ge c\sqrt{\Theta(x)\!-\!\Theta(\overline{x})}.
  \]

  \noindent
  {\bf Subcase 2.2:} $x\notin{\rm dom}\partial\Theta$. In this case, $\partial\Theta(x)=\emptyset$,
  and $\mbox{dist}(0,\partial\Theta(x))=\infty$. This means that
  the last inequality automatically holds.

  Now by the arbitrariness of $x$ in
  $\mathbb{B}(\overline{x},\delta)\cap [\Theta(\overline{x})<\Theta<\Theta(\overline{x})+ \eta]$,
  the last inequality shows that the function $\Theta$ has the KL property of exponent $1/2$
  at $\overline{x}$. By the arbitrariness of $\overline{x}$ in ${\rm crit}\Theta$,
  the desired result follows.  \qed
 \end{proof}

 By Remark \ref{KL-remark}, Theorem \ref{KLzconstr-Theta} shows that
  $\Theta$ with $h\equiv\delta_{\Omega}$ is a KL function of exponent $1/2$
  if the associated $\theta$ satisfies Assumption \ref{assump1}.
  Thus, when $h\equiv\delta_{\Omega}$, the function $\Theta$
  associated to those $f$ and $\theta$ in Example \ref{example41}-\ref{example45}
  has the KL property of exponent $1/2$ at every point of ${\rm crit}\Theta$.

  \begin{remark}\label{remark-theorem}
   Consider the function $\Theta$ involves $\widetilde{h}:=h+\delta_{\mathbb{R}_{+}^p}$
   instead of $h$. Suppose that Assumption \ref{assump1} (iii) is replaced by the following condition:
   \begin{description}
     \item[(iii')] for every $x\in{\rm dom}\partial\theta$,
               $\partial(\theta+\widetilde{h})(x)\subseteq\!\partial\theta(x)\!+\partial\widetilde{h}(x)$.
  \end{description}
  Then, by following the proofs of Theorem \ref{KLznorm-Theta} and \ref{KLzconstr-Theta}
  and using the relation $\partial\widetilde{h}(\overline{x})
  \subseteq \partial h(\overline{x})+\partial\delta_{\mathbb{R}_{+}^p}(\overline{x})
  \subseteq \partial h(\overline{x})$ by Proposition \ref{prop32},
  it is not hard to show that the conclusions of two theorems still hold.
  Thus, the function $\Theta$ involves $\widetilde{h}$ and
  those those $f$ and $\theta$ in Example \ref{example41}-\ref{example45}
  is still a KL function of exponent $1/2$, for example, the function
  \(
    \Theta(x)\!:=\!x^{\mathbb{T}}Ax +\delta_{\mathcal{S}\cap\mathbb{R}_{+}^p}(x)+h(x).
  \)
 \end{remark}

 To close this section, we demonstrate the linear convergence phenomenon
 of the proximal gradient method (PGM) and the Nesterov's accelerated
 proximal gradient method for solving \eqref{Theta} with $\Theta$ given by
 Example \ref{example41} and \ref{example43}. Among others, the $\Theta$ in
 Example \ref{example43} is using $\widetilde{h}=\nu\|\cdot\|_0+\delta_{\mathbb{R}_{+}^p}$.
 For Example \ref{example43}, we generate the data $A,b$ in the same way
 as \cite[Section 4.1]{Wen17}; and for Example \ref{example41} we first
 generate randomly the sample matrix $X\in\mathbb{R}^{p\times n}$ whose
 each column obeys the distribution $N(0,\Sigma)$ and take $A=X^{\mathbb{T}}X$.
 Among others, the covariance matrix $\Sigma$ is generated as follows:
 let $\Sigma'$ with $\Sigma_{ij}'=0.5^{|i-j|}$ have the eigenvalue decomposition
 $\Sigma'=Q{\rm Diag}(\lambda(\Sigma'))Q^{\mathbb{T}}$, replace
 the first $k=30$ columns of $Q$ by sparse eigenvalue vectors generated randomly,
 and then set $\Sigma=Q{\rm Diag}(\lambda(\Sigma'))Q^{\mathbb{T}}$.
 Figure \ref{fig1} plots the iterate error curve, i.e.,
 the distance curve from the iterates to the final output $x^{\rm out}$ of two solvers.
 We see that the iterate sequence $\{x^k\}$ is indeed linearly convergent,
 and for the more difficult nonnegative zero-norm regularized logistic regression problem,
 the APG is remarkably superior to the PGM. This confirms the obtained result.
\begin{figure}[ht]
\begin{center}
\includegraphics[width=4.9 in]{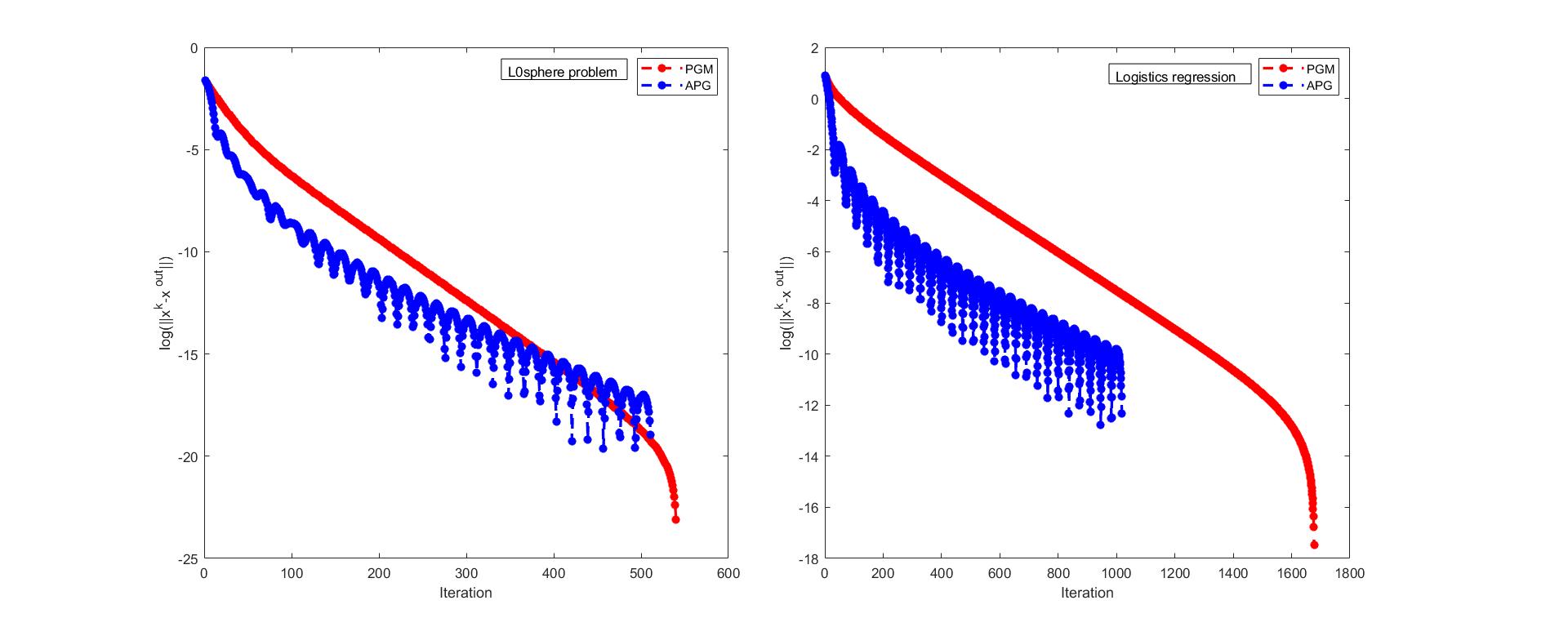}
\end{center}
\caption{The iterate errors yielded by PGM and APG for solving Example \ref{example41} and \ref{example43} }
\label{fig1}
\end{figure}

 \section{Conclusions}\label{sec6}

 Since the family of KL functions of exponent $1/2$ is lack of the stability,
 their identification is not an easy task even for convex functions.
 In this paper, we have established the KL property of exponent $1/2$ for
 the zero-norm regularized and constrained composite function $\Theta$,
 provided that the involved proper lsc function $\theta$ satisfies
 Assumption \ref{assump1}. Some specific examples for $\theta$ are
 also provided to show that such an assumption can be satisfied.
 Obviously, the obtained results are also applicable to
 matrix optimization problems which involve the loss function of
 matrix entries and the sparsity of matrix entries. Our future work will
 focus on this property for the matrix optimization problems involving
 row or column sparsity or the sparsity of singular value vectors.

\begin{acknowledgements}
{\renewcommand\baselinestretch{1.0}\selectfont
 The authors would like to give their sincere thanks to two anonymous reviewers
 for their helpful comments, which improve greatly the original manuscript.
 The research of S. H. Pan and S. J. Bi is supported by the National Natural
 Science Foundation of China under project No.11971177 and No.11701186,
 and Guangdong Basic and Applied Basic Research Foundation (2020A1515010408).
\par}
\end{acknowledgements}

\titleformat{\section}[display]
{\Large\normalsize\bfseries}{ }{0pt}{\Large}
\titleformat{\subsection}[display]
{\normalfont\normalsize\bfseries}{ }{0pt}{\normalsize}

\renewcommand{\appendixname}{ \Alph{subsection} }

\appendix
\section{Appendices}
\subsection{\appendixname :  KL Property Relative to a Manifold}
\setcounter{alemma}{0}
\renewcommand\thealemma{\Alph{subsection}.\arabic{alemma}}

 Let $\mathcal{M}\subset\mathbb{R}^p$ be a $\mathcal{C}^2$-smooth manifold
  and $f\!:\mathcal{M}\to\mathbb{R}$ be a $\mathcal{C}^2$-smooth function.
  The set of critical points of the problem
  \(
    \min_{x\in\mathcal{M}}f(x)
  \)
  is $\mathcal{X}:=\big\{x\in\mathcal{M}:\ \nabla_{\!\mathcal{M}}f(x)=0\big\}$,
  where $\nabla_{\!\mathcal{M}}f(z)$ is the projection of $\nabla\!f(z)$
  onto the tangent space $\mathcal{T}_{\mathcal{M}}(z)$ of $\mathcal{M}$ at $z$.
  We say that $f$ is a KL function of exponent $1/2$ relative to $\mathcal{M}$
  if $f$ has the KL property of exponent $1/2$ at each $\overline{x}\in\mathcal{X}$,
  i.e., there exist $\delta>0$ and $\gamma>0$ such that
  \begin{equation*}\label{gradf}\tag{13}
    \|\nabla_{\!\mathcal{M}}f(z)\| \ge \gamma\sqrt{|f(z)-f(\overline{x})|}
    \quad\ \forall z\in\mathbb{B}(\overline{x},\delta)\cap\mathcal{M}.
  \end{equation*}
  This part states the relation between the KL property of exponent $1/2$ of $f$
  relative to $\mathcal{M}$ and the KL property of exponent $1/2$ for
  its extended $\widetilde{f}(x):=f(x)+\delta_{\mathcal{M}}(x)$ for $x\in\mathbb{R}^p$.
 \begin{lemma}\label{manifold-KL}
  Let $\mathcal{M}\!\subset\!\mathbb{R}^p$ be a $\mathcal{C}^2$-smooth manifold
  and $f\!:\mathcal{M}\to\mathbb{R}$ be a $\mathcal{C}^2$-smooth function.
  If $f$ is a KL function of exponent $1/2$ relative to $\mathcal{M}$,
  then $\widetilde{f}$ is a KL function of exponent $1/2$. Conversely,
  if $\widetilde{f}$ is a KL function of exponent $1/2$ and each critical point
  is a local minimizer, then $f$ is a KL function of exponent $1/2$ relative to $\mathcal{M}$.
  \end{lemma}
 \begin{proof}
  Notice that $\partial\!\widetilde{f}(x)=\nabla\!f(x)+\mathcal{N}_{\mathcal{M}}(x)$
  for any $x\in\mathcal{M}$. Clearly, $\mathcal{X}={\rm crit}\widetilde{f}$.
  Fix an arbitrary $\overline{x}\in\mathcal{X}$. Since $f$ has the KL property
  of exponent $1/2$ relative to $\mathcal{M}$ at $\overline{x}$,
  there exist $\delta>0$ and $\gamma>0$ such that \eqref{gradf} holds
  for all $z\in\mathbb{B}(\overline{x},\delta)\cap\mathcal{M}$.
  Fix an arbitrary $\eta>0$ and an arbitrary $x\in\mathbb{B}(\overline{x},\delta)
  \cap[\widetilde{f}(\overline{x})<\widetilde{f}<\widetilde{f}(\overline{x})+\eta]$.
  Clearly, $x\in\mathcal{M}$. Moreover,
  \begin{equation*}\label{relation-f}\tag{14}
   {\rm dist}(0,\partial\!\widetilde{f}(x))
   =\|\nabla f(x)-\Pi_{\mathcal{N}_{\mathcal{M}}(x)}(\nabla\!f(x))\|
   =\|\Pi_{\mathcal{T}_{\mathcal{M}}(x)}(\nabla\!f(x))\|=\|\nabla\!f_{\mathcal{M}}(x)\|.
  \end{equation*}
  Along with \eqref{gradf}, ${\rm dist}(0,\partial\!\widetilde{f}(x))\ge\gamma\sqrt{f(x)-f(\overline{x})}$.
  So, the first part of the results follows.

  Next we focus on the second part. Fix an arbitrary $\overline{x}\in\mathcal{X}$.
  By the given assumption, clearly, $\overline{x}$ is a local optimal solution of
  \(
    \min_{x\in\mathcal{M}}f(x).
  \)
  Hence,  there exists $\varepsilon'>0$ such that
  \[
    f(z)\ge f(\overline{x})\quad\ \forall z\in\mathbb{B}(\overline{x},\varepsilon')\cap\mathcal{M}.
  \]
  By the KL property of exponent $1/2$ of $\widetilde{f}$ at $\overline{x}$,
  there exist $\varepsilon,c>0$ and $\eta>0$ such that
  \begin{equation*}\label{wf-ineq}\tag{15}
   {\rm dist}(0,\partial\!\widetilde{f}(x))\ge c\sqrt{f(x)-\widetilde{f}(x)}
   \quad\forall x\in\mathbb{B}(\overline{x},\varepsilon)
  \cap[\widetilde{f}(\overline{x})<\widetilde{f}<\widetilde{f}(\overline{x})+\eta].
  \end{equation*}
  Since $f$ is $\mathcal{C}^2$-smooth around $\overline{x}$,
  there exists $\varepsilon''>0$ such that for all
  $z\in\mathbb{B}(\overline{x},\varepsilon'')\cap\mathcal{M}$,
  \(
    f(z)<f(\overline{x})+\eta.
  \)
  Take $\delta=\min(\varepsilon,\varepsilon',\varepsilon'')$. Fix an arbitrary
  $x\in\mathbb{B}(\overline{x},\delta)\cap\mathcal{M}$. Clearly,
  \(
    f(\overline{x})\le f(x)\le f(\overline{x})+\eta.
  \)
  If $f(x)>f(\overline{x})$, then $x\in\mathbb{B}(\overline{x},\varepsilon)
  \cap[\widetilde{f}(\overline{x})<\widetilde{f}<\widetilde{f}(\overline{x})+\eta]$,
  and from \eqref{wf-ineq} and \eqref{relation-f},
  \(
    \|\nabla\!f_{\mathcal{M}}(x)\| \ge c\sqrt{|f(x)-f(\overline{x})|}.
  \)
  If $f(x)=f(\overline{x})$, this inequality holds automatically. \qed
 \end{proof}
 \subsection{\appendixname : KL Property of the Quadratic Function over a Sphere}
\setcounter{alemma}{0}
\renewcommand\thealemma{\Alph{subsection}.\arabic{alemma}}

 For any integer $m\ge 1$ and any given $m\times m$ real symmetric $H$, define
 $g(z):=z^{\mathbb{T}}Hz+\delta_{\mathcal{S}}(z)$ for $z\in\mathbb{R}^m$.
 Lemma 1 in Appendix A and \cite[Theorem 1]{Liu18} imply that $g$ is
 a KL function of exponent $1/2$. This part gives a different
 proof, which needs the following lemmas.
 \begin{alemma}\label{Spoint-g}
  The critical point set of $g$ takes the form of
  \(
    {\rm crit}g=\big\{z\in\mathcal{S}:\ Hz=\langle z,Hz\rangle z\big\}.
  \)
  So, by letting $H$ have the eigenvalue decomposition $P\Lambda P^{\mathbb{T}}$
  with $\Lambda={\rm diag}(\lambda_1,\ldots,\lambda_m)$
  for $\lambda_1\ge\cdots\ge\lambda_m$ and $P\in\mathbb{O}^m$,
  ${\rm crit}g =PW$ with
  \(
    W=\big\{y\in\mathcal{S}:\ \Lambda y=\langle y,\Lambda y\rangle y\big\}.
  \)
 \end{alemma}
 \begin{proof}
  By \cite[Exercise 8.8]{RW98} and Lemma \ref{Gsubdiff-sphere},
  it immediately follows that for any $z\in\mathbb{R}^m$,
  \begin{equation*}\label{g-subdiff}\tag{16}
   \partial g(z)=2Hz+\partial\delta_{\mathcal{S}}(z)=2Hz+[\![z]\!].
 \end{equation*}
  Choose an arbitrary $\overline{z}\in{\rm crit}g$.
  From \eqref{g-subdiff}, there exists $\overline{t}\in\mathbb{R}$ such that
  $0=2H\overline{z}+\overline{t}\overline{z}$. Along with $\|\overline{z}\|=1$,
  we have $\overline{t}=-2\langle \overline{z},H\overline{z}\rangle$,
  and hence $\overline{z}\in\!\big\{z\in\mathcal{S}:\ Hz=\!\langle z,Hz\rangle z\big\}$.
  Consequently,
  \(
    {\rm crit}g\subseteq\big\{z\in\mathcal{S}:\ Hz=\langle z,Hz\rangle z\big\}.
  \)
  The converse inclusion is immediate to check by Lemma \ref{Gsubdiff-sphere}.
  Thus, the first part follows. The second part is immediate.  \qed
 \end{proof}
 \begin{alemma}\label{lemma-psi}
  Let $D={\rm diag}(d_1,d_2,\ldots,d_p)$ with $d_1\ge d_2\ge\cdots\ge d_p$.
  Define the function $\psi(x):=x^{\mathbb{T}}Dx+\delta_S(x)$
  for $x\in\mathbb{R}^p$. Then, $\psi$ is a KL function of exponent $1/2$.
 \end{alemma}
 \begin{proof}
  By Lemma \ref{Spoint-g} it is immediate to obtain the following
  characterization for ${\rm crit}\psi$:
  \begin{equation*}\label{crit-psi}\tag{17}
    {\rm crit}\,\psi=\big\{x\in\mathcal{S}:\ Dx=\langle x,Dx\rangle x\big\}.
  \end{equation*}
  Clearly, for each $x\in{\rm crit}\,\psi$,
  $d_i=\langle x,Dx\rangle$ with $i\in{\rm supp}(x)$.
  For any $z\in{\rm dom}\,\partial\psi$, we have
  \begin{equation*}\label{equa1-appendix}\tag{18}
  \begin{aligned}
  \mbox{dist}(0,\partial\psi(z))^2
  &= \min_{u\in \partial\psi(z)} \|u\|^2= \min_{w\in\mathbb{R}} \|2Dz + wz\|^2\\
  &=\min_{w\in\mathbb{R}}\Big\{4\langle z,D^{\mathbb{T}}Dz\rangle+w^2+4w\langle z,Dz\rangle\Big\}\\
  &= 4\langle z,D^{\mathbb{T}}Dz\rangle-4(\langle z,Dz\rangle)^2
  =4\|Dz-\langle z,Dz\rangle z\|^2.
 \end{aligned}
 \end{equation*}
 Now fix an arbitrary $\overline{x}\in{\rm crit}\,\psi$.
 From \eqref{crit-psi} it immediately follows that
 \(
   -D\overline{x}+\langle \overline{x},D\overline{x}\rangle\overline{x}=0.
 \)
 We next proceed the arguments by two cases as will be shown below.

 \noindent
 {\bf Case 1: $d_1=\cdots=d_p=\gamma$ for some $\gamma\in\mathbb{R}$}.
 Choose an arbitrary $\eta>0$ and an arbitrary $\delta>0$.
 Fix an arbitrary $x \in \mathbb{B}(\overline{x},\delta)\cap[\psi(\overline{x})<\psi(x) < \psi(\overline{x})+\eta]$.
 Clearly, $x\in\mathcal{S}$ and $\langle x,Dx\rangle=\gamma$.
 Combining $\langle \overline{x},D\overline{x}\rangle\overline{x}=D\overline{x}$
 and equation \eqref{equa1-appendix} yields that
 \[
   \mbox{dist}(0, \partial\psi(x))
   =4\|Dx-\langle x,Dx\rangle x-(D\overline{x}-\langle \overline{x},D\overline{x}\rangle\overline{x})\|
   =0.
 \]
 In addition, $\psi(x)=\psi(\overline{x})=\gamma$.
 This means that
 \(
   \mbox{dist}(0, \partial\psi(x))=\sqrt{\psi(x)-\psi(\overline{x})}.
 \)

 \noindent
 {\bf Case 2: there exist $i\ne j\in\{1,2,\ldots,p\}$ such that $d_i\ne d_j$}.
 Write $J={\rm supp}(\overline{x})$ and $\overline{J}=\{1,\ldots,p\}\backslash J$.
 By \eqref{crit-psi}, we know that $d_i=\langle\overline{x},D\overline{x}\rangle$
 for all $i\in J$. This means that there must exist an index $\kappa\in\overline{J}$
 such that $d_{\kappa}\ne\langle\overline{x},D\overline{x}\rangle$. Write
 \(
   \overline{J}_1:=\big\{i\in\overline{J}:\ d_i\ne\langle\overline{x},D\overline{x}\rangle\big\}.
 \)
 By the continuity of the function $\langle \cdot,D\cdot\rangle$, there exists $\delta>0$
 such that for all $z\in\mathbb{B}(\overline{x},\delta)\cap\mathcal{S}$,
 \begin{equation*}\label{same-order}\tag{19}
  \frac{1}{2}|d_j-\langle\overline{x},D\overline{x}\rangle|\le |d_j-\langle z,Dz\rangle|\le \frac{3}{2}|d_j-\langle\overline{x},D\overline{x}\rangle|\quad\forall j\in\overline{J}_1.
 \end{equation*}
 Choose an arbitrary $\eta>0$. Fix an arbitrary
 $x\in\mathbb{B}(\overline{x},\delta)\cap[\psi(\overline{x})<\psi(x)<\psi(\overline{x})+\eta]$.
 Clearly, $x \in \mathcal{S}$. From equation \eqref{equa1-appendix},
 it follows that
 \begin{align*}\label{ineq1-appendix}\tag{20}
   &\frac{1}{4}\mbox{dist}(0,\partial\psi(x))^2
   =\sum_{j\in\overline{J}}\big(d_j-\langle x,Dx\rangle\big)^2x_j^2
    +\sum_{j\in J}\big(d_j-\langle x,Dx\rangle\big)^2x_j^2\nonumber\\
   &=\sum_{j\in\overline{J}}\big(d_j-\langle x,Dx\rangle\big)^2x_j^2
    +\sum_{j\in J}\big(\langle\overline{x},D\overline{x}\rangle-\langle x,Dx\rangle\big)^2x_j^2\nonumber\\
   &\ge \sum_{j\in\overline{J}_1}\big(d_j-\langle x,Dx\rangle\big)^2x_j^2
   \ge \frac{1}{4}\sum_{j\in\overline{J}_1}\big(d_j-\langle\overline{x},D\overline{x}\rangle\big)^2x_j^2
 \end{align*}
 where the third equality is due to \eqref{crit-psi},
 the first inequality is by the definition of $\overline{J}_1$,
 and the last inequality is due to \eqref{same-order}. On the other hand,
 by the definition of $\psi$,
 \begin{align*}\label{ineq2-appendix}\tag{21}
  \psi(x)-\psi(\overline{x})&=\langle x,Dx\rangle-\langle\overline{x},D\overline{x}\rangle
  =\sum_{j\in\overline{J}}d_jx_j^2+\sum_{j\in J}d_jx_j^2-\langle\overline{x},D\overline{x}\rangle\|x\|^2\nonumber\\
  &=\sum_{j\in\overline{J}}\big(d_j-\langle\overline{x},D\overline{x}\rangle\big)x_j^2
    +\sum_{j\in J}\big(d_j-\langle\overline{x},D\overline{x}\rangle\big)x_j^2\nonumber\\
  &=\sum_{j\in\overline{J}}\big(d_j-\langle\overline{x},D\overline{x}\rangle\big)x_j^2
  =\sum_{j\in\overline{J}_1}\big(d_j-\langle\overline{x},D\overline{x}\rangle\big)x_j^2\nonumber\\
  &\le\sum_{j\in\overline{J}_1}|\langle\overline{x},D\overline{x}\rangle-d_j| x_j^2
  \le \max_{j\in\overline{J}_1}|d_j-\langle\overline{x},D\overline{x}\rangle|\|x_{\overline{J}_1}\|^2
 \end{align*}
 where the fourth equality is due to \eqref{crit-psi},
 the fifth one is by the definition of $\overline{J}_1$, and the inequality is since
 $\psi(x)-\psi(\overline{x})>0$. From the above inequalities \eqref{ineq1-appendix}
 and \eqref{ineq2-appendix},
 \begin{align*}
  \mbox{dist}(0, \partial\psi(x))&\ge\sqrt{\sum_{j\in\overline{J}_1}\big(d_j-\langle\overline{x},D\overline{x}\rangle\big)^2x_j^2}
  \ge \min_{j\in\overline{J}_1}|d_j-\langle\overline{x},D\overline{x}\rangle|\|x_{\overline{J}_1}\|\\
  &\ge\frac{\min_{j\in\overline{J}_1}|d_j-\langle\overline{x},D\overline{x}\rangle|}
  {\sqrt{\max_{j\in\overline{J}_1}|d_j-\langle\overline{x},D\overline{x}\rangle|}}\sqrt{\psi(x)-\psi(\overline{x})}.
 \end{align*}

 By the arbitrariness of $x$, Case 1 and 2 show that $\psi$ has the KL property with exponent $1/2$
 at $\overline{x}$. From the arbitrariness of $\overline{x}$ in ${\rm crit}\psi$,
 $\psi$ is a KL function of exponent $1/2$. \qed
 \end{proof}

 Now we prove that $g$ is a KL function of exponent $1/2$.
 Fix an arbitrary $\overline{z}\in{\rm crit}g$. Let $H$ have the eigenvalue
 decomposition as in Lemma \ref{Spoint-g}. Then $\overline{y}=P^{\mathbb{T}}\overline{z}\in{\rm crit}\psi$
  where $\psi$ is defined in Lemma \ref{lemma-psi} with $D=\Lambda$.
  By Lemma \ref{lemma-psi}, there exist $\eta>0,\delta>0$ and $c>0$ such that
  \[
    {\rm dist}(0,\partial\psi(y))\ge c\sqrt{\psi(y)-\psi(\overline{y})}
    \quad\ \forall y\in\mathbb{B}(\overline{y},\delta)\cap[\psi(\overline{y})<\psi<\psi(\overline{y})+\eta].
  \]
  Fix an arbitrary $z\in\mathbb{B}(\overline{z},\delta)\cap[g(\overline{z})<g<g(\overline{z})+\eta]$.
  Clearly, $z\in\mathcal{S}$. Write $y=P^{\mathbb{T}}z$. Then $y\in\mathcal{S}$ and
  $g(z)=\psi(y)$. Since $g(\overline{z})=g(\overline{y})$,
  \(
    y\in\mathbb{B}(\overline{y},\delta)\cap[\psi(\overline{y})<\psi(y) < \psi(\overline{y})+\eta].
  \)
  In addition, from \eqref{g-subdiff} and the eigenvalue decomposition of $H$,
  $\partial g(z)=P\partial\psi(y)$. Thus,
  \[
   {\rm dist}(0,\partial g(z))={\rm dist}(0,P\partial\psi(y))
   ={\rm dist}(0,\partial \psi(y))\ge c\sqrt{\psi(y)-\psi(\overline{y})}.
  \]
  Together with $\psi(y)-\psi(\overline{y})=g(z)-g(\overline{z})$,
  it follows that $g$ has the KL property with exponent of 1/2 at $\overline{z}$.
  By the arbitrariness of $\overline{z}$ in ${\rm crit}g$,
  $g$ is a KL function of exponent $1/2$.
 \subsection{\appendixname :  Supplementary Lemma and Proofs}
\setcounter{alemma}{0}
\renewcommand\thealemma{\Alph{subsection}.\arabic{alemma}}

  The following lemma extends the result of \cite[Section 2.3]{Karimi16}
  for the differentiable strongly convex function to the setting of
  closed proper strongly convex functions. In particular, it implies that
  the composite $g$ is a KL function of exponent $1/2$ without surjectivity of $\mathcal{A}$.
 \begin{alemma}\label{composite}
  Consider $g(x)\!:=\vartheta(\mathcal{A}x)$ for $x\in\mathbb{X}$
  where $\mathcal{A}\!:\mathbb{X}\to\mathbb{Z}$ is a linear mapping,
  and $\vartheta\!:\mathbb{Z}\to ]-\infty,\infty]$ is a proper closed
  strongly convex function with modulus $\mu$. Here, $\mathbb{X}$ and
  $\mathbb{Z}$ are two finite dimensional vector spaces equipped with
  the inner product $\langle\cdot,\cdot\rangle$ and its induced norm $\|\cdot\|$.
  If ${\rm ri}({\rm dom}\vartheta)\cap{\rm range}\mathcal{A}\ne\emptyset$,
  then there exists a constant $\overline{c}>0$ such that
  \begin{equation*}\label{aim-eq2}\tag{22}
   {\rm dist}(0, \partial g(x))\ge\frac{\sqrt{2\mu}}{\overline{c}}
   \sqrt{g(x)-g^*}\quad\ \forall x\in\mathbb{X}
  \end{equation*}
  where $g^*$ denotes the minimum value of the function $g$.
 \end{alemma}
 \begin{proof}
  Pick an arbitrary $x^*\in{\rm crit}g$ (if ${\rm crit}g=\emptyset$,
  the conclusion holds automatically). We first prove that
  ${\rm crit}g =\{x\in\mathbb{X}\!:\,\mathcal{A}x=\mathcal{A}x^*\}$.
  To this end, pick any $x'\in\mathbb{X}$ with $\mathcal{A}x'=\mathcal{A}x^*$.
  Since ${\rm ri}({\rm dom}\vartheta)\cap{\rm range}\mathcal{A}\ne\emptyset$,
  by \cite[Theorem 23.9]{Roc70}, we have
  \(
   \partial g(x')=\mathcal{A}^*\partial \vartheta(\mathcal{A}x')
   = \mathcal{A}^*\partial \vartheta(\mathcal{A}x^*) = \partial g(x^*).
  \)
  Notice that $0\in \partial g(x^*) $, we obtain $0\in \partial g(x')$,
  which implies that $x'\in{\rm crit}g$. This means that
  $\{x\in\mathbb{X}\!:\,\mathcal{A}x=\mathcal{A}x^*\}\subseteq{\rm crit}g$.
  Suppose that there exist $\overline{x}\in {\rm crit} g$ such that
  $\mathcal{A}\overline{x}\ne \mathcal{A}x^*$. Then, by the strong convexity of $\vartheta$,
  we have
  \[
 	g((\overline{x} + x^*)/2) = \vartheta ((\mathcal{A}\overline{x}+\mathcal{A}x^*)/2)
    <(g(\overline{x})+ g(x^*))/2.
  \]
  This contradicts the fact that $x', x^* \in {\rm crit}g$.
  Thus, the equality ${\rm crit} g= \{x \in \mathbb{X}\!:\,\mathcal{A}x = \mathcal{A}x^*\}$
  holds. By Hoffman inequality \cite{Hoffman52},
  there exist a constant $\overline{c}>0$ such that for any $z \in \mathbb{X}$,
  \begin{equation*}\label{eq1}\tag{23}
 	\| \Pi_{{\rm crit} g} (z) -z\| \leq \overline{c} \|\mathcal{A}(\Pi_{{\rm crit} g} (z) - z)\|,
  \end{equation*}
  where $\Pi_{{\rm crit} g}$ is the projection mapping onto ${\rm crit} g$.
  Fix an arbitrary $x\in\mathbb{X}$. If $x\notin{\rm dom}\partial g$,
  the inequality \eqref{aim-eq2} holds trivially. So, it suffices to consider
  the case $x\in{\rm dom}\partial g$. By \cite[Theorem 23.9]{Roc70},
  $\partial g(x)=\mathcal{A}^* \partial \vartheta (\mathcal{A}x)$.
  Obviously, $\partial \vartheta (\mathcal{A}x) \neq \emptyset$.
  Pick any $\xi \in \partial \vartheta (\mathcal{A}x)$. By the strong convexity of
  $\vartheta$ and \cite[Theorem 6.1.2]{Lem91}, it follows that
  \[
 	g(z) \geq g(x)+\langle\xi,\mathcal{A}(z-x)\rangle+\frac{\mu}{2}\|\mathcal{A}(z-x)\|^2
     \quad\ \forall z\in \mathbb{X}.
  \]
  By taking $z = \Pi_{{\rm crit} g}(x)$, from the last inequality we obtain that
  \begin{equation*}
  \begin{aligned}
 	g(\Pi_{{\rm crit} g}(x))
   &\ge g(x) + \langle \xi, \mathcal{A}(\Pi_{{\rm crit} g}(x)-x) \rangle
        + \frac{\mu}{2} \|\mathcal{A}(\Pi_{{\rm crit} g}(x)-x)\|^2 \\
   &\ge g(x) + \langle \xi, \mathcal{A}(\Pi_{{\rm crit} g}(x)-x)\rangle
      + \frac{\mu}{2\overline{c}^2} \|\Pi_{{\rm crit} g}(x)-x\|^2 \\
   &\ge g(x) + \min_{y\in\mathbb{X}}\left[\langle \xi, \mathcal{A}(y-x)\rangle
        + \frac{\mu}{2\overline{c}^2} \|y-x\|^2 \right]
   \ge g(x) - 0.5 (\overline{c}^2/\mu) \|\mathcal{A}^*\xi\|^2,
  \end{aligned}
  \end{equation*}
  where the second inequality follows from \eqref{eq1}.
  Note that $g(\Pi_{{\rm crit} g}(x))=g(x^*)$. The last inequality
  implies that $\|\mathcal{A}^*\xi\|^2\geq (2\mu/\overline{c}^2)[g(x)-g(x^*)]$.
  Together with $\partial g(x)=\mathcal{A}^*\partial\vartheta (\mathcal{A}x)$,
  \[
    {\rm dist}(0, \partial g(x))^2\ge\min_{\xi\in\partial\vartheta (\mathcal{A}x)}\|\mathcal{A}^*\xi\|^2
    \geq (2\mu/\overline{c}^2)[g(x)-g(x^*)].
  \]
  This implies that the desired inequality \eqref{aim-eq2} holds. \qed
 \end{proof}

 \noindent
 {\bf The proof of Proposition \ref{prop32}:}
  First, we assume that $\psi$ is a proper closed piecewise linear regular function.
  Fix any $\overline{x}\in\mathbb{R}^p$ with $\partial\psi(\overline{x})\ne\emptyset$.
  Notice that ${\rm epi}\psi$ and ${\rm epi}h$ are the union of finitely many polyhedral sets.
  Together with \cite[Proposition 1]{Robinson81} and \cite[Section 3.2]{Ioffe08},
  it follows that
  \begin{equation*}
    \partial(\psi+h)(\overline{x})\subseteq \partial\psi(\overline{x})+\partial h(\overline{x})
    \ \ {\rm and}\ \
    \partial^{\infty}(\psi+h)(\overline{x})\subseteq \partial^{\infty}\psi(\overline{x})+\partial^{\infty} h(\overline{x}).
  \end{equation*}
  By combining the two inclusions with \cite[Corollary 10.9]{RW98}
  and the regularity of $\psi$ and $h$, we conclude that $\psi+h$ are regular,
  and moreover, it holds that
  \[
     \widehat{\partial}(\psi+h)
     =\partial(\psi+h)(\overline{x})=\partial\psi(\overline{x})+\partial h(\overline{x}),\
     \partial^{\infty}(\psi+h)(\overline{x})=\partial^{\infty}\psi(\overline{x})+\partial^{\infty} h(\overline{x}).
  \]
  The first group of equalities imply that $\partial(\psi+h)(\overline{x})\ne \emptyset$
  since $\partial\psi(\overline{x})\ne\emptyset$.
  Thus, by \cite[Corollary 8.11]{RW98},
  \(
    \partial^{\infty}(\psi+h)(\overline{x})
    =[\widehat{\partial}(\psi+h)(\overline{x})]^{\infty}
    =[\partial\psi(\overline{x})+\partial h(\overline{x})]^{\infty}.
  \)

  Now we assume that $\psi=\delta_C$. Fix any $\overline{x}\in C$ with
  ${\rm ri}(C)\cap L_{\overline{x}}\ne\emptyset$,
  where $L_{\overline{x}}\!:=\{x\in\mathbb{R}^p\,|\, x_i=0\ {\rm for}\ i\notin{\rm supp}(\overline{x})\}$.
  Write $J={\rm supp}(\overline{x}),\overline{J}=\{1,\ldots,p\}\backslash J$.
  By Lemma 3.2, $\partial h(\overline{x})=\mathcal{N}_{L_{\overline{x}}}(\overline{x})$.
  We first argue that
  \begin{equation}\label{aim-inclusion}\tag{24}
    \widehat{\partial}(\delta_C\!+h)(\overline{x})
     \subseteq\partial\delta_{C\cap L_{\overline{x}}}(\overline{x}).
  \end{equation}

  \noindent
  {\bf Case 1: there exists $\widehat{x}\in[C\cap L_{\overline{x}}]\backslash\{\overline{x}\}$.}
  Pick any $v\in\widehat{\partial}(\delta_C+h)(\overline{x})$.
  By the definition of regular subgradient, it follows that
  \begin{align*}
  0&\le\liminf_{x'\rightarrow\overline{x}, x'\neq \overline{x}}\frac{h(x')+\delta_C(x')
       -h(\overline{x})-\delta_C(\overline{x})-\langle v,x'-\overline{x}\rangle}{\|x'-\overline{x}\|}\\
  &\le\liminf_{\overline{x}\ne x'\xrightarrow[C]{}\overline{x}\atop{\rm supp}(x')=J}
      \frac{h(x')-h(\overline{x})-\langle v,x'-\overline{x}\rangle}{\|x'-\overline{x}\|}
   =\liminf_{\overline{x}\ne x'\xrightarrow[C]{}\overline{x}\atop{\rm supp}(x')=J}
  \frac{-\langle v,x'-\overline{x}\rangle}{\|x'-\overline{x}\|}\\
  &=\liminf_{\overline{x}\ne x'\xrightarrow[C]{}\overline{x}\atop{\rm supp}(x')=J}
  \frac{\delta_{C\cap L_{\overline{x}}}(x')-\delta_{C\cap L_{\overline{x}}}(\overline{x})
      -\langle v,x'-\overline{x}\rangle}{\|x'-\overline{x}\|}
 \end{align*}
 where the existence of $\overline{x}\ne x'\xrightarrow[C]{}\overline{x}$ with ${\rm supp}(x')=J$
 in the second inequality is implied by $\widehat{x}\in[C\cap L_{\overline{x}}]\backslash\{\overline{x}\}$.
 The last group of inequalities imply that $v\in\widehat{\partial}\delta_{C\cap L_{\overline{x}}}(\overline{x})
 =\partial\delta_{C\cap L_{\overline{x}}}(\overline{x})$.
 By the arbitrariness of $v\in\widehat{\partial}(\delta_C+h)(\overline{x})$,
 it immediately follows that
 $\widehat{\partial}(\delta_C+h)(\overline{x})\subseteq\partial\delta_{C\cap L_{\overline{x}}}(\overline{x})$.

 \medskip
 \noindent
 {\bf Case 2: $C\cap L_{\overline{x}}=\{\overline{x}\}$.}
 Now $\partial\delta_{C\cap L_{\overline{x}}}(\overline{x})=\mathcal{N}_{C\cap L_{\overline{x}}}(\overline{x})=\mathbb{R}^p$.
 Then, it is immediate to have that
 $\widehat{\partial}(\delta_C+h)(\overline{x})\subseteq\partial\delta_{C\cap L_{\overline{x}}}(\overline{x})$.

 By combining \eqref{aim-inclusion} with \cite[Corollary 10.9]{RW98} and Lemma 3.2, it holds that
  \begin{equation*}\label{equa4}\tag{25}
  \begin{aligned}
   \partial\delta_C(\overline{x})+\partial h(\overline{x})=
    \widehat{\partial}\delta_C(\overline{x})+\widehat{\partial}h(\overline{x})
    \subseteq\widehat{\partial}(\delta_C+h)(\overline{x})\subseteq\partial(\delta_C+\delta_{L_{\overline{x}}})(\overline{x})\\
    =\partial\delta_C(\overline{x})+\partial\delta_{L_{\overline{x}}}(\overline{x})
   =\partial\delta_C(\overline{x})+\partial h(\overline{x}).
  \end{aligned}
  \end{equation*}
  where the second equality is due to ${\rm ri}C \cap L_{\overline{x}}\ne\emptyset$
  and the last one is by $\partial h(\overline{x})=\mathcal{N}_{L_{\overline{x}}}(\overline{x})$.
  In fact, from the above arguments, we conclude that
  \begin{equation}\label{equa5}\tag{26}
    \widehat{\partial}(\delta_{C}\!+\!h)(x)
    =\partial\delta_C(x)+\partial h(x)=\partial\delta_{C\cap L_{x}}(x)
  \end{equation}
  for all $x\in C\ {\rm with}\ {\rm ri}(C)\cap L_{x}\ne\emptyset$.

   Next we argue that $\partial(\delta_C\!+h)(\overline{x})\subseteq
  \partial\delta_C(\overline{x})+\partial h(\overline{x})$.
  Pick any $v\in\partial(\delta_C+h)(\overline{x})$. Then,
  there exist sequences $x^k\xrightarrow[\delta_{C}+h]{}\overline{x}$ and
  $v^k\in\widehat{\partial}(\delta_{C}\!+\!h)(x^k)$
  with $v^k\to v$ as $k\to\infty$. Since $\delta_{C}(x^k)+h(x^k)\to
  \delta_{C}(\overline{x})+h(\overline{x})$, we must have $x^k\in C$
  and $h(x^k)\to h(\overline{x})$ for all sufficiently large $k$.
  The latter, along with ${\rm supp}(x^k)\supseteq J$, implies that
  ${\rm supp}(x^k)=J$ for all sufficiently large $k$.
  From the last equality, for all sufficiently large $k$,
  \(
    v^k\in\partial\delta_{C}(x^k)+\partial h(x^k).
  \)
  By passing to the limit $k\to \infty$ and using $h(x^k)\to h(\overline{x})$,
  we obtain $v\in\partial\delta_{C}(\overline{x})+\partial h(\overline{x})$.
  By the arbitrariness of $v$ in $\partial(\delta_C+h)(\overline{x})$,
  the stated inclusion follows. In particular, together with
  $\partial(\delta_C\!+h)(\overline{x})\supseteq\widehat{\partial}(\delta_C\!+h)(\overline{x})
  =\partial\delta_C(\overline{x})+\partial h(\overline{x})$
  and \eqref{equa4},
  \begin{equation}\label{equa6}\tag{27}
    \widehat{\partial}(\delta_C\!+h)(\overline{x})
     =\partial(\delta_C\!+h)(\overline{x})
     =\mathcal{N}_{C}(\overline{x})+\partial h(\overline{x})
     =\partial\delta_{C\cap L_{\overline{x}}}(\overline{x}).
  \end{equation}

  Next we argue that $\partial^{\infty}(\delta_C\!+h)(\overline{x})=
  \partial^{\infty}\delta_{C\cap L_{\overline{x}}}(\overline{x})$.
  Pick any $u\in\partial^{\infty}(\delta_C\!+h)(\overline{x})$. Then,
  there exist sequences $x^k\xrightarrow[\delta_{C}+h]{}\overline{x}$ and
  $u^k\in\widehat{\partial}(\delta_{C}\!+\!h)(x^k)$
  with $\lambda_ku^k\to u$ for some $\lambda_k\downarrow 0$ as $k\to\infty$.
  By following the same arguments as above, we have
  ${\rm supp}(x^k)=J$ for all sufficiently large $k$.
  Together with \eqref{equa5} and $u^k\in\widehat{\partial}(\delta_{C}\!+\!h)(x^k)$,
  for all sufficiently large $k$ we have
  \(
    u^k\in\widehat{\partial}\delta_{C\cap L_{\overline{x}}}(x^k).
  \)
  Notice that $x^k\xrightarrow[C\cap L_{\overline{x}}]{}\overline{x}$.
  So, $u\in\partial^{\infty}(\delta_C\!+\!\delta_{L_{\overline{x}}})(\overline{x})
  =\partial^{\infty}\delta_{C\cap L_{\overline{x}}}(\overline{x})$.
  By the arbitrariness of $u$ in $\partial^{\infty}(\delta_C\!+h)(\overline{x})$,
  it holds that $\partial^{\infty}(\delta_C\!+h)(\overline{x})\subseteq
  \partial^{\infty}\delta_{C\cap L_{\overline{x}}}(\overline{x})$.
  Conversely, pick any $u\in\partial^{\infty}\delta_{C\cap L_{\overline{x}}}(\overline{x})$.
  Then, there exist sequences $x^k\xrightarrow[C\cap L_{\overline{x}}]{}\overline{x}$ and
  $u^k\in\widehat{\partial}\delta_{C\cap L_{\overline{x}}}(x^k)$
  with $\lambda_ku^k\to u$ for some $\lambda_k\downarrow 0$ as $k\to\infty$.
  Clearly, $(\delta_C+h)(x^k)\to(\delta_C+h)(\overline{x})$.
  Moreover, from \eqref{equa5} and $u^k\in\widehat{\partial}\delta_{C\cap L_{\overline{x}}}(x^k)$,
  we have $u^k\in\widehat{\partial}(\delta_{C}+h)(x^k)$.
  This shows that $u\in\partial^{\infty}(\delta_C+h)(\overline{x})$.
  By the arbitrariness of $u$ in $\partial^{\infty}\delta_{C\cap L_{\overline{x}}}(\overline{x})$,
  we have the converse inclusion $\partial^{\infty}(\delta_C\!+h)(\overline{x})\supseteq
  \partial^{\infty}\delta_{C\cap L_{\overline{x}}}(\overline{x})$.
  Thus, the stated equality follows.
  From \cite[Exercise 8.14 \& Proposition 8.12]{RW98},
  $\partial\delta_{C\cap L_{\overline{x}}}(\overline{x})
  =\partial^{\infty}\delta_{C\cap L_{\overline{x}}}(\overline{x})
   =[\widehat{\partial}\delta_{C\cap L_{\overline{x}}}(\overline{x})]^{\infty}$.
  Thus,
  \[
   \partial\delta_{C\cap L_{\overline{x}}}(\overline{x})
   =\partial^{\infty}\delta_{C\cap L_{\overline{x}}}(\overline{x})
   =[\widehat{\partial}\delta_{C\cap L_{\overline{x}}}(\overline{x})]^{\infty}
   =\partial^{\infty}(\delta_C\!+h)(\overline{x}).
  \]
  Together with equalities in \eqref{equa6}, we obtain the conclusion
  for $h(\cdot)=\nu\|\cdot\|_0$ and the regularity of $\delta_C\!+h$.
  By following the same arguments as above, one may obtain the second part,
  and we omit the details.
  \qed


\end{document}